\documentclass[12pt]{amsart}

\usepackage{ucs}

\usepackage{amssymb}
\usepackage{amsthm}
\usepackage{amsmath}
\usepackage{latexsym}
\usepackage[cp1251]{inputenc}
\usepackage{graphicx}
\usepackage{wrapfig}
\usepackage{caption}
\usepackage{subcaption}
\usepackage{indentfirst}
\usepackage[left=2.6cm,right=2.6cm,top=2.6cm,bottom=2.6cm,bindingoffset=0cm]{geometry}
\usepackage{enumerate}
\usepackage{makecell}

\DeclareMathOperator{\aut}{Aut}

\DeclareMathOperator{\cay}{Cay}

\DeclareMathOperator{\rk}{rk}

\DeclareMathOperator{\Span}{Span}

\DeclareMathOperator{\sym}{Sym}
\DeclareMathOperator{\rad}{rad}

\DeclareMathOperator{\dimwl}{dim_{WL}}
\DeclareMathOperator{\rkwl}{rk_{WL}}
\DeclareMathOperator{\WL}{WL}

\def\r{\mathrm{right}}

\makeatletter 
\def\@seccntformat#1{\csname the#1\endcsname. } 
\def\@biblabel#1{#1.}

\makeatother

\title{On WL-rank of Deza Cayley graphs}

\author{Dmitry Churikov}
\address{Sobolev Institute of Mathematics, Novosibirsk, Russia}
\address{Novosibirsk State University, Novosibirsk, Russia}
\email{churikovdv@gmail.com}

\author{Grigory Ryabov}
\address{Sobolev Institute of Mathematics, Novosibirsk, Russia}
\address{Novosibirsk State University, Novosibirsk, Russia}
\email{gric2ryabov@gmail.com}

\thanks{The work is supported by the Mathematical Center in Akademgorodok under the agreement No. 075-15-2019-1613 with the Ministry of Science and Higher Education of the Russian Federation.}

\date{}

\newtheorem{prop}{Proposition}[section]

\newtheorem*{theo1}{Theorem 1}

\newtheorem{lemm}[prop]{Lemma}

\theoremstyle{definition}

\newtheorem*{rem}{Remark}

\begin{document}

\vspace{\baselineskip}
\vspace{\baselineskip}

\vspace{\baselineskip}

\vspace{\baselineskip}

\begin{abstract}
The \emph{WL-rank} of a digraph $\Gamma$ is defined to be the rank of the coherent configuration of $\Gamma$. We construct a new infinite family of strictly  Deza Cayley graphs for which the WL-rank is equal to the number of vertices. The graphs from this family are divisible design and integral.
\end{abstract}

\maketitle

\section{Introduction}

Let $V$ be a finite set and $|V|=n$. A \emph{coherent configuration} $\mathcal{X}$ on $V$ can be thought as a special partition of $V\times V$ for which the diagonal of $V\times V$ is a union of classes (see~\cite[Definition~2.1.3]{CP}). The number of classes is called the \emph{rank} of $\mathcal{X}$. Let $\Gamma=(V,E)$ be a digraph with vertex set $V$ and arc set $E$. The \emph{WL-rank} (the \emph{Weisfeiler-Leman rank}) of $\Gamma$ is defined to be the rank of the smallest coherent configuration on the set $V$ for which $E$ is a union of classes. The term ``WL-rank of a digraph'' was introduced in~\cite{BPR}. This term was chosen because the coherent configuration of a digraph can be found using the Weisfeiler-Leman algorithm~\cite{WeisL}. Since the diagonal of $V\times V$ is a union of classes of any coherent configuration on $V$, we conclude that $\rkwl(\Gamma)\geq 2$ unless $|V|=1$. One can verify that $\rkwl(\Gamma)\leq 2$ if and only if $\Gamma$ is complete or empty. On the other hand, obviously, $\rkwl(\Gamma)\leq n^2$. From~\cite[Lemma~2.1 (2)]{BPR} it follows that if $\Gamma$ is vertex-transitive then $\rkwl(\Gamma)\leq n$.

Let $G$ be a finite group, $|G|=n$, and $S$ an identity-free subset of $G$. The \emph{Cayley digraph} $\cay(G,S)$ is defined to be the digraph with vertex set $G$ and arc set $\{(g,sg):~s\in S,~g\in G\}$. If $S$ is inverse-closed then $\cay(G,S)$ is a \emph{Cayley graph}. If $\Gamma$ is a Cayley digraph over $G$ then $\aut(\Gamma)\geq G_{\r}$, where $G_{\r}$ is the subgroup of $\sym(G)$ induced by right multiplications of $G$. This implies that $\Gamma$ is vertex-transitive and hence $\rkwl(\Gamma)\leq n$.

A $k$-regular graph $\Gamma$ is called  \emph{strongly regular} if there exist nonnegative integers $\lambda$ and $\mu$ such that every two adjacent vertices have $\lambda$ common neighbors and every two nonadjacent vertices have $\mu$ common neighbors. The following generalization of the notion of a strongly regular graph was introduced in~\cite{EFHHH} and goes back to~\cite{Deza}. A $k$-regular graph $\Gamma$ on $n$ vertices is called a \emph{Deza} graph if there exist nonnegative integers $\alpha$ and $\beta$ such that any pair of distinct vertices of $\Gamma$ has either $\alpha$ or $\beta$ common neighbors. The numbers $(n,k,\beta,\alpha)$ are called the \emph{parameters} of $\Gamma$. Clearly, if $\alpha>0$ and $\beta>0$ then $\Gamma$ has diameter~$2$. A Deza graph is called a \emph{strictly} Deza graph if it is nonstrongly regular and has diameter~$2$.

The WL-rank of a strongly regular graph is at most~$3$ (see~\cite[Lemma~2.1 (2)]{BPR}). It is a natural question how large the WL-rank of a Deza graph $\Gamma$ can be. In this paper we are interested in the WL-rank of Deza Cayley graphs. The WL-rank of a nonstrictly Deza Cayley graph can be sufficiently large. For example, an undirected cycle on $n$ vertices is a nonstrictly Deza graph of WL-rank~$[\frac{n}{2}]+1$ (see~\cite{BPR}). However, strictly Deza graphs seem close to strongly regular graphs. All known strictly Deza Cayley graphs over cyclic groups have WL-rank at most~$6$~\cite{BPR}. As it was said before, the WL-rank of any Cayley graph does not exceed the number of vertices of this graph. It turns out that there exists an infinite family of strictly Deza Cayley graphs whose WL-rank is equal to the number of vertices. This follows from the theorem below which is the main result of this paper. The cyclic and dihedral groups of order $n$ are denoted by $C_n$ and $D_n$ respectively.

\begin{theo1}\label{main}
Let $k\geq 3$ be an odd integer, $G\cong D_{2k}\times C_2\times C_2$, and $n=|G|$. There exists a strictly Deza Cayley graph $\Gamma$ over $G$ such that $\rkwl(\Gamma)=n$.
\end{theo1}

Note that the graphs from Theorem~\ref{main} are divisible design integral graphs (see Section~$4$).

We finish the introduction with the brief outline of the paper. If $\Gamma=\cay(G,S)$ then the WL-rank of $\Gamma$ is equal to the rank of the smallest $S$-ring over $G$ for which $S$ is a union of basic sets. The necessary background of $S$-rings and Cayley graphs is provided in Section~$2$. In Section~$3$ we construct the required family of strictly Deza Cayley graphs and prove Theorem~\ref{main}. In Section~$4$ we prove that each graph from the constructed family is an integral divisible design graph (Lemma~\ref{l4}), has the same parameters as the grid graph but not isomorphic to it (Lemma~\ref{l5}), and can be identified efficiently (Lemma~\ref{l6}).

The authors would like to thank prof. I.~Ponomarenko for the valuable comments which help us to improve the text significantly.

\section{Preliminaries}
In this section we provide a background of $S$-rings and Cayley graphs. In general, we follow to~\cite{BPR,Ry1,Ry2}, where the most of  definitions and statements is contained.

\subsection{$S$-rings}

Let $G$ be a finite group and $\mathbb{Z}G$  the integer group ring. The identity element of $G$ and the set of all nonidentity elements of $G$ are denoted by~$e$ and~$G^\#$ respectively. If $X\subseteq G$ then the element $\sum \limits_{x\in X} {x}$ of the group ring $\mathbb{Z}G$ is denoted by~$\underline{X}$. An easy straightforward computation implies that $\underline{G}^2=|G|\underline{G}$. The set $\{x^{-1}:x\in X\}$ is denoted by $X^{-1}$.

A subring  $\mathcal{A}\subseteq \mathbb{Z} G$ is called an \emph{$S$-ring} (a \emph{Schur} ring) over $G$ if there exists a partition $\mathcal{S}=\mathcal{S}(\mathcal{A})$ of~$G$ such that:

$(1)$ $\{e\}\in\mathcal{S}$;

$(2)$  if $X\in\mathcal{S}$ then $X^{-1}\in\mathcal{S}$;

$(3)$ $\mathcal{A}=\Span_{\mathbb{Z}}\{\underline{X}:\ X\in\mathcal{S}\}$.

\noindent The notion of an $S$-ring goes back to Schur~\cite{Schur} and Wielandt~\cite{Wi}.

The elements of $\mathcal{S}$ are called the \emph{basic sets} of  $\mathcal{A}$ and the number $\rk(\mathcal{A})=|\mathcal{S}|$ is called the \emph{rank} of~$\mathcal{A}$. The group ring $\mathbb{Z}G$ is an $S$-ring over $G$ corresponding to the partition of $G$ into singletons and $\rk(\mathbb{Z}G)=|G|$.

The following lemma provides a well-known property of $S$-rings (see, e.g.~\cite[Lemma~2.4]{Ry1}).

\begin{lemm}\label{basicset}
Let $\mathcal{A}$ be an $S$-ring over a group $G$. If $X,Y\in \mathcal{S}(\mathcal{A})$ then $XY\in \mathcal{S}(\mathcal{A})$ whenever $|X|=1$ or $|Y|=1$.
\end{lemm}

\begin{lemm}\label{groupring}
Let $\mathcal{A}$ be an $S$-ring over a group $G$ and $X\subseteq G$ such that $\langle X \rangle=G$. Suppose that $\{x\}\in \mathcal{S}(\mathcal{A})$ for every $x\in X$. Then $\mathcal{A}=\mathbb{Z}G$.
\end{lemm}

\begin{proof}
Let us prove that $\{g\}\in \mathcal{S}(\mathcal{A})$ for every $g\in G$. Since $\langle X \rangle=G$, there exist $x_1,\ldots,x_k\in X$ and $\varepsilon_1,\ldots \varepsilon_k\in\{-1,1\}$ such that $g=x_1^{\varepsilon_1}\ldots x_k^{\varepsilon_k}$. We proceed by induction on $k$. Let $k=1$. If $\varepsilon_1=1$ then $\{g\}\in \mathcal{S}(\mathcal{A})$ by the assumption of the lemma; if $\varepsilon_1=-1$ then $\{g\}\in \mathcal{S}(\mathcal{A})$ by the assumption of the lemma and the second property from the definition of an $S$-ring. Now let $k\geq 2$. By the induction hypothesis, we have $\{x_1^{\varepsilon_1}\ldots x_{k-1}^{\varepsilon_{k-1}}\}\in \mathcal{S}(\mathcal{A})$ and $\{x_k^{\varepsilon_k}\}\in \mathcal{S}(\mathcal{A})$. So $\{g\}=\{x_1^{\varepsilon_1}\ldots x_{k-1}^{\varepsilon_{k-1}}\}\{x_k^{\varepsilon_k}\}\in \mathcal{S}(\mathcal{A})$ by Lemma~\ref{basicset}.
\end{proof}

A set $X \subseteq G$ is called an \emph{$\mathcal{A}$-set} if $\underline{X}\in \mathcal{A}$ or, equivalently, $X$ is a union of some basic sets of $\mathcal{A}$. The set of all $\mathcal{A}$-sets is denoted by $\mathcal{S}^*(\mathcal{A})$. Obviously, if $X\in \mathcal{S}^*(\mathcal{A})$ and $|X|=1$ then $X\in \mathcal{S}(\mathcal{A})$. It is easy to check that if $X,Y\in \mathcal{S}^*(\mathcal{A})$ then 
$$X\cap Y,X\cup Y, X\setminus Y, Y\setminus X, XY\in \mathcal{S}^*(\mathcal{A}).~\eqno(1)$$ 
A subgroup $H \leq G$ is called an \emph{$\mathcal{A}$-subgroup} if $H\in \mathcal{S}^*(\mathcal{A})$. For every $\mathcal{A}$-set $X$, the groups $\langle X \rangle$ and $\rad(X)=\{g\in G:~Xg=gX=X\}$ are $\mathcal{A}$-subgroups. 

\begin{lemm}\cite[Proposition~22.1]{Wi}\label{sw}
Let $\mathcal{A}$ be an $S$-ring over $G$, $\xi=\sum \limits_{g\in G} c_g g\in \mathcal{A}$, where $c_g\in \mathbb{Z}$, and $c\in \mathbb{Z}$. Then $\{g\in G:~c_g=c\}\in \mathcal{S}^*(\mathcal{A})$.
\end{lemm}

Let $L \unlhd U\leq G$. A section $U/L$ is called an \emph{$\mathcal{A}$-section} if $U$ and $L$ are $\mathcal{A}$-subgroups. If $S=U/L$ is an $\mathcal{A}$-section then the module
$$\mathcal{A}_S=Span_{\mathbb{Z}}\left\{\underline{X}^{\pi}:~X\in\mathcal{S}(\mathcal{A}),~X\subseteq U\right\},$$
where $\pi:U\rightarrow U/L$ is the canonical epimorphism, is an $S$-ring over $S$.

Let $S=U/L$ be an $\mathcal{A}$-section of $G$. The $S$-ring~$\mathcal{A}$ is called the \emph{$S$-wreath product} or \emph{generalized wreath product} of $\mathcal{A}_U$ and $\mathcal{A}_{G/L}$ if $L\trianglelefteq G$ and $L\leq\rad(X)$ for each basic set $X$ outside~$U$. In this case we write $\mathcal{A}=\mathcal{A}_U\wr_{S}\mathcal{A}_{G/L}$. If $L>\{e\}$ and $U<G$ then the $S$-wreath product is called \emph{nontrivial}. The notion of the generalized wreath product of $S$-rings was introduced in~\cite{EP1}. Since $L\leq\rad(X)$ for each basic set $X$ outside~$U$, the basic sets of $\mathcal{A}$ outside $U$ are in one-to-one correspondence with the basic sets of $\mathcal{A}_{G/L}$ outside $S$. Therefore
$$\rk(\mathcal{A}_U\wr_{S}\mathcal{A}_{G/L})=\rk(\mathcal{A}_U)+\rk(\mathcal{A}_{G/L})-\rk(\mathcal{A}_S).~\eqno(2)$$

The \emph{automorphism group} $\aut(\mathcal{A})$ of $\mathcal{A}$ is defined to be the group 
$$\bigcap \limits_{X\in \mathcal{S}(\mathcal{A})} \aut(\cay(G,X)).$$
Since $\aut(\cay(G,X))\geq G_{\r}$ for every $X\in \mathcal{S}(\mathcal{A})$, we conclude that $\aut(\mathcal{A})\geq G_{\r}$. It is easy to check that $\aut(\mathcal{A})=G_{\r}$ if and only if $\mathcal{A}=\mathbb{Z}G$.

\subsection{Cayley graphs}

Let $S\subseteq G$, $e\notin S$, and $\Gamma=\cay(G,S)$. The \emph{WL-closure} $\WL(\Gamma)$ of $\Gamma$ can be thought as the smallest $S$-ring over $G$ such that $S\in\mathcal{S}^*(\mathcal{A})$ (see~\cite[Section~5]{BPR}). If $\mathcal{A}=\WL(\Gamma)$ then $\rkwl(\Gamma)=\rk(\mathcal{A})$ by~\cite[Lemma~5.1]{BPR}. From~\cite[Theorem~2.6.4]{CP} it follows that $\aut(\Gamma)=\aut(\mathcal{A})$.

\begin{lemm}\cite[Lemma~5.2]{BPR}\label{deza}
Let $G$ be a group of order~$n$, $S\subseteq G$ such that $e\notin S$, $S=S^{-1}$, and $|S|=k$, and $\Gamma=\cay(G,S)$. The graph $\Gamma$ is a Deza graph with parameters $(n,k,\beta,\alpha)$ if and only if $\underline{S}^2=ke+\alpha\underline{X_{\alpha}}+\beta\underline{X_{\beta}}$, where $X_{\alpha}\cup X_{\beta}=G^\#$ and $X_{\alpha}\cap X_{\beta}=\varnothing$. Moreover, $\Gamma$ is strongly regular if and only if $X_{\alpha}=S$ or $X_{\beta}=S$.
\end{lemm}

\section{Proof of Theorem~1}

Let $k\geq 3$ be an integer, $G=(\langle a \rangle \rtimes \langle b \rangle)\times \langle c \rangle \times \langle d \rangle$, where $|a|=k$, $|b|=|c|=|d|=2$, and $bab=a^{-1}$, and $n=|G|$. The groups $\langle a \rangle$, $\langle c\rangle$, and $\langle a \rangle \rtimes \langle b \rangle$ are denoted by $A$, $C$, and $H$ respectively. Clearly, $H\cong D_{2k}$, $G\cong D_{2k}\times C_2 \times C_2$, $|H|=2k$, and $|G|=8k$. Put 
$$S=b(A\setminus \{a^{-1}\})\cup c(A\cup\{b\})\cup \{db,dcba^{-1}\}.$$
One can see that $S=S^{-1}$ and $|S|=2(k+1)$. Put $\Gamma=\cay(G,S)$. Note that $\Gamma$ is $2(k+1)$-regular.

\begin{lemm}\label{l1}
In the above notations, the graph $\Gamma$ is a strictly Deza graph with parameters $(8k,2(k+1),2(k-1),2)$.
\end{lemm}

\begin{proof}
The straightforward computation in the group ring $\mathbb{Z}G$ using the equalities $\underline{A}^2=k\underline{A}$, $b\underline{A}=\underline{A}b$, $bab=a^{-1}$, $cg=gc$, and $dg=gd$, where $g\in G$, implies that
$$\underline{S}^2=2(k+1)e+2(k-1)(\underline{A}^\#+cb\underline{A})+2(b+c)\underline{A}+2d\underline{C}\underline{H}.~\eqno(3)$$
Indeed,
$$\underline{S}^2=(b\underline{A}+c\underline{A}-ba^{-1}+cb+db+dcba^{-1})^2=$$
$$=4e+(2(k-1)e+2(k-1)cb+2b+2c+2d+2dc+2db+2dcb)\underline{A}=$$
$$=2(k+1)e+2(k-1)(\underline{A}^\#+cb\underline{A})+2(b+c)\underline{A}+2d\underline{C}\underline{H}.$$

From Lemma~\ref{deza} and Eq.~(3) it follows that $\Gamma$ is a nonstrongly regular Deza graph with parameters $(8k,2(k+1),2(k-1),2)$, $X_{2(k-1)}=A^\#\cup cbA$, and $X_2=bA\cup cA\cup d(C\times H)$. This means that $\Gamma$ is a strictly Deza graph.
\end{proof}

All Deza Cayley graphs with at most~$60$ vertices, including the graphs from the constructed family for $k\leq 7$, were enumerated in~\cite{GSh}.

Put $A_1=\langle a^2 \rangle$. If $k$ is odd then $A_1=A$; if $k$ is even then $|A:A_1|=2$. The group $A_1$ is normal in $G$. So one can form the group $L=A_1\rtimes \langle cb \rangle$ which is isomorphic to $D_{2k}$ if $k$ is odd and to $D_k$ if $k$ is even. It can be verified in a straightforward way that $L$ is normal in $G$. Put $U=\langle L,ca,da\rangle$ and $S=U/L$. Since $L\cap \langle ca \rangle=L\cap \langle da \rangle=A_1$, we obtain $|U:L|=4$.

\begin{lemm}\label{l2}
In the above notations, $\WL(\Gamma)=\mathbb{Z}G$ if $k$ is odd and $\WL(\Gamma)=\mathbb{Z}U\wr_S \mathbb{Z}(G/L)$ if $k$ is even.
\end{lemm}

\begin{proof}
Let $\mathcal{A}=\WL(\Gamma)$. Put $V=A^\#\cup cbA$. From Eq.~(3) it follows that every element of $V$ enters the element $\underline{S}^2$ with coefficient~$2(k-1)$ and any other element of $G$ enters $\underline{S}^2$ with coefficient distinct from~$2(k-1)$. Together with $S\in \mathcal{S}^*(\mathcal{A})$ and Lemma~\ref{sw}, this implies that $V\in \mathcal{S}^*(\mathcal{A})$. So 
$$V\cap S=\{cb\}\in \mathcal{S}(\mathcal{A})~\eqno(4)$$
by Eq.~(1). Since $S,\{cb\}\in \mathcal{S}^*(\mathcal{A})$, Eq.~(1) implies that $cbS,Scb\in \mathcal{S}^*(\mathcal{A})$. So
$$S_1=(cbS\setminus Scb)\cap S=\{ca\}\in \mathcal{S}(\mathcal{A})~\eqno(5)$$
by Eq.~(1). Now from Eqs.~(1) and~(5) it follows that
$$S_2=(cbS\setminus Scb)\setminus S_1=\{da^{-1}\}\in \mathcal{S}(\mathcal{A})~\eqno(6)$$
Due to Eqs.~(1) and~(5), we obtain $S_1S_1=\{a^2\}\in \mathcal{S}(\mathcal{A})$. Since $a^2$ is a generator of $A_1$, Lemma~\ref{groupring} yields that 
$$\mathcal{A}_{A_1}=\mathbb{Z}A_1.~\eqno(7)$$

Let $k$ be odd. Then $A_1=A$ and $G=\langle A, cb, ca, da^{-1} \rangle$. From Eqs.~(4)-(7) and Lemma~\ref{groupring} it follows that $\mathcal{A}=\mathbb{Z}G$. 

Let $k$ be even. The partition of $G$ into sets
$$\{g\},~g\in U,~La,~Lc,~Ld,~Lcda$$
defines the $S$-ring $\mathcal{B}$ over $G$ such that $\mathcal{B}=\mathbb{Z}U\wr_S \mathbb{Z}(G/L)$. Note that $S=Lc\cup S_U$, where 
$$S_U=b(A\setminus (A_1\cup \{a^{-1}\}))\cup c((A\setminus A_1)\cup \{b\})\cup \{db,dcba^{-1}\}\subseteq U.$$
So $S\in \mathcal{S}^*(\mathcal{B})$ and hence $\mathcal{B}\geq \mathcal{A}$. 

Let us prove that $\mathcal{B}\leq \mathcal{A}$. Observe that $da\in da^{-1}A_1$. So 
$$\{da\}\in \mathcal{S}(\mathcal{A})~\eqno(8)$$ 
by Eqs.~(6)-(7) and Lemma~\ref{basicset}. Eqs.~(4),~(5),~(7),~(8), and Lemma~\ref{groupring} imply that  $U$ is an $\mathcal{A}$-subgroup and
$$\mathcal{A}_U=\mathbb{Z}U=\mathcal{B}_U.~\eqno(9)$$
Since $S\in \mathcal{S}^*(\mathcal{A})$ and $U\in \mathcal{S}^*(\mathcal{A})$, Eq.~(1) implies that 
$$S\setminus U=Lc\in \mathcal{S}^*(\mathcal{A}).~\eqno(10)$$ 
From Eqs. (1),~(5),~(6),~(8), and~(10) it follows that 
$$La=Lc\{ca\},~Ld=Lc\{ca\}\{da^{-1}\},~Lcda=Lc\{da\}\in\mathcal{S}^*(\mathcal{A}).$$
Together with Eq.~(9), this implies that every basic set of $\mathcal{B}$ is an $\mathcal{A}$-set and hence $\mathcal{B}\leq \mathcal{A}$. Thus, $\mathcal{B}=\mathcal{A}$ and we are done.
\end{proof}

\begin{rem}
If $k$ is odd then $\aut(\Gamma)=\aut(\mathbb{Z}G)=G_{\r}$. If $k$ is even then $\aut(\Gamma)=\aut(\mathbb{Z}U\wr_S \mathbb{Z}(G/L))$ is the canonical generalized wreath product of $U_{\r}$ by $(G/L)_{\r}$ (see~\cite[Section~5.3]{EP2} for the definitions).
\end{rem}

\begin{lemm}\label{l3}
In the above notations, $\rkwl(\Gamma)=8k=n$ if $k$ is odd and $\rkwl(\Gamma)=4k+4=\frac{n}{2}+4$ if $k$ is even.
\end{lemm}

\begin{proof}
If $k$ is odd then $\WL(\Gamma)=\mathbb{Z}G$ by Lemma~\ref{l2} and hence $\rkwl(\Gamma)=\rk(\mathbb{Z}G)=8k$. Let $k$ be even. Then $\WL(\Gamma)=\mathbb{Z}U\wr_S \mathbb{Z}(G/L)$ by Lemma~\ref{l2}. Since $|L|=k$ and $|U:L|=4$, we have $|U|=4k$ and hence $\rk(\mathbb{Z}U)=4k$. Observe that $|G/L|=8$ and $|S|=|U/L|=4$. So $\rk(\mathbb{Z}(G/L))=8$ and $\rk(\mathbb{Z}S)=4$. Therefore 
$$\rkwl(\Gamma)=\rk(\mathbb{Z}U\wr_S \mathbb{Z}(G/L))=\rk(\mathbb{Z}U)+\rk(\mathbb{Z}(G/L))-\rk(\mathbb{Z}S)=4k+4$$
by Eq.~(2).
\end{proof}

Theorem~\ref{main} follows from Lemma~\ref{l1} and Lemma~\ref{l3}.

\section{Some properties of $\Gamma$}

In this section we collect some properties of the graph $\Gamma$ constructed in the previous section. 

A $k$-regular graph on $n$ vertices is called a \emph{divisible design graph} (\emph{DDG}) with parameters $(n,k,\alpha,\beta,m,l)$ if its vertex set can be partitioned into $m$ classes of size $l$, such that every two distinct vertices from the same class have $\alpha$ common neighbors and every two vertices from different classes have $\beta$ common neighbors. For a divisible design graph, the partition into classes is called a \emph{canonical partition}. The notion of a divisible design graph was introduced in~\cite{HKM} as a generalization of $(v,k,\lambda)$-graphs~\cite{Rud}. For more information on divisible design graphs, we refer the readers to~\cite{HKM,KS}.

A graph is called \emph{integral} if all eigenvalues of its adjacency matrix are integers. The investigations on integral graphs goes back to~\cite{HS}. More information on spectra of graphs and integral graphs can be found in~\cite{BH}.

\begin{lemm}\label{l4}
The graph $\Gamma$ is an integral divisible design graph.
\end{lemm}

\begin{proof}
From~\cite[Theorem~1.1]{KS} it follows that $\Gamma$ is a divisible design graph if and only if, in the notations of Lemma~\ref{deza}, $X_2\cup \{e\}$ or $X_{2(k-1)}\cup \{e\}$ is a subgroup of $G$. Moreover, the canonical partition of $G$ is a partition into the right cosets by this subgroup. Eq.~(3) implies that $X_{2(k-1)}\cup\{e\}=A\cup cbA$. Since $A$ is normal in $G$, $|cb|=2$, and $a^{cb}=a^{-1}$, the set $X_{2(k-1)}\cup\{e\}$ is a subgroup of $G$ isomorphic to $D_{2k}$. Therefore $\Gamma$ is a divisible design graph with parameters $(8k,2(k+1),2(k-1),2,4,2k)$. 

Since $\Gamma$ is a divisible design graph, one can calculate eigenvalues of its adjacency matrix from its parameters by the formulas from~\cite[Lemma~2.1]{HKM}. It turns out that the set of eigenvalues of the adjacency matrix of $\Gamma$ is equal to $\{2(k+1),\pm 2(k-1),\pm 2\}$. This implies that $\Gamma$ is integral.
\end{proof}

Recall that the \emph{$(l\times m)$-grid} is the line graph of the complete bipartite graph $K_{l,m}$ (see~\cite[p.~440]{BKN}).

\begin{lemm}\label{l5}
The graph $\Gamma$ has the same parameters as the $(4 \times 2k)$-grid  but it is not isomorphic to the $(4 \times 2k)$-grid.
\end{lemm}

\begin{proof}
Let $\Gamma^{\prime}$ be the graph isomorphic to the $(4\times 2k)$-grid. The graph $\Gamma^{\prime}$ has parameters $(8k,2(k+1),2(k-1),2)$ by~\cite[Construction~4.8]{HKM}. However, due to~\cite[Example~3.2.12]{CP}, we obtain $\rkwl(\Gamma^{\prime})=4$ and $\aut(\Gamma^{\prime})\cong \sym(4) \times \sym(2k)$. So $\Gamma$ is not isomorphic to $\Gamma^{\prime}$.
\end{proof}

The \emph{Weisfeiler-Leman dimension} $\dimwl(\Delta)$ of a graph $\Delta$ is defined to be the smallest positive integer~$m$ for which $\Delta$ is identified by the $m$-dimensional Weisfeiler-Leman algorithm~~\cite[Definition~18.4.2]{Grohe}. If $\dimwl(\Delta)\leq m$ then the isomorphism between $\Delta$ and any other graph can be verified in time $n^{O(m)}$ using the Weisfeiler-Leman algorithm~\cite{WeisL}. The Weisfeiler-Leman dimension of Deza circulant graphs was studied in~\cite{BPR}.

\begin{lemm}\label{l6}
The Weisfeiler-Leman dimension of $\Gamma$ is equal to~$2$.
\end{lemm}

\begin{proof}
The $S$-ring $\WL(\Gamma)$ is separable in the sense of~\cite[Section~4.2]{BPR}. Indeed, if $k$ is odd then $\WL(\Gamma)=\mathbb{Z}G$ by Lemma~\ref{l2} and the required follows from~\cite[Theorem~2.3.33]{CP}. If $k$ is even then $\WL(\Gamma)=\mathbb{Z}U\wr_S \mathbb{Z}(G/L)$ by Lemma~\ref{l2} and the required follows from~\cite[Theorem~3.4.23]{CP}. The separability of $\WL(\Gamma)$ and~\cite[Theorem~2.5]{FKV} imply that $\dimwl(\Gamma)\leq 2$. Since $\Gamma$ is regular but nonstrongly regular, $\dimwl(\Gamma)\neq 1$ by~\cite[Lemma~3.2]{BPR}. Thus, $\dimwl(\Gamma)=2$.
\end{proof}

\end{document}